\documentclass{amsart}

\usepackage{tikz}
\usepackage{xcolor}

\usepackage{latexsym,amsmath,amssymb,epic}
\usepackage{latexsym,amssymb,epic}
\usepackage{amsmath,amsthm}
\usepackage{color}

\usepackage{soul}

\usepackage{hyperref}
\hypersetup{
	colorlinks=true, 
	linktoc=all, 
	linkcolor=blue} 

\newtheorem{theorem}{Theorem}[section]
\newtheorem{lemma}[theorem]{Lemma}
\newtheorem{corollary}[theorem]{Corollary}

\newtheorem{conjecture}[theorem]{Conjecture}

\allowdisplaybreaks

\begin{document}
	
\title[New Congruences for 2--Regular Partitions with Designated Summands]{New Infinite Families of Congruences Modulo Powers of 2 for 2--Regular Partitions with Designated Summands}

\author{James A. Sellers}
\address{Department of Mathematics and Statistics, University of Minnesota Duluth, Duluth, MN 55812, USA}
\email{jsellers@d.umn.edu}

	
	
\maketitle
\begin{abstract}
In 2002, Andrews, Lewis, and Lovejoy introduced the combinatorial objects which they called {\it partitions with designated summands}.  These are built by taking unrestricted integer partitions and designating exactly one of each occurrence of a part.  In that same work, Andrews, Lewis, and Lovejoy also studied such partitions wherein all parts must be odd.  Recently, Herden, Sepanski, Stanfill, Hammon, Henningsen, Ickes, and Ruiz proved a number of Ramanujan--like congruences for the function $PD_2(n)$ which counts the number of partitions of weight $n$ with designated summands wherein all parts must be odd.  In this work, we prove some of the results conjectured by Herden, et. al. by proving the following two infinite families of congruences satisfied by $PD_2(n)$:  For all $\alpha\geq 0$ and $n\geq 0,$
\begin{eqnarray*}
PD_2(2^\alpha(4n+3)) &\equiv & 0 \pmod{4} \ \ \ \ \ {\textrm and} \\
PD_2(2^\alpha(8n+7)) &\equiv & 0 \pmod{8}.
\end{eqnarray*}
All of the proof techniques used herein are elementary, relying on classical $q$--series identities and generating function manipulations.  
\end{abstract}

\section{Introduction}  
In 2002, Andrews, Lewis, and Lovejoy \cite{ALL} introduced the combinatorial objects which they called {\it partitions with designated summands}.  These are built by taking unrestricted integer partitions and designating exactly one of each occurrence of a part.  For example, there are 10 partitions with designated summands of weight 4:  

$$
4',\ \ \  3'+1',\ \ \  2'+2,\ \ \  2+2',\ \ \  2'+1'+1,\ \ \  2'+1+1'
$$
$$
1'+1+1+1,\ \ \  1+1'+1+1,\ \ \  1+1+1'+1,\ \ \  1+1+1+1'
$$
Andrews, Lewis, and Lovejoy denoted the number of partitions with designated summands of weight $n$ by the function $PD(n).$  Hence, using this notation and the example above, we know $PD(4)=10.$

In the same paper, Andrews, Lewis, and Lovejoy \cite{ALL} also considered the restricted partitions with designated summands wherein all parts must be odd, and they denoted the corresponding enumeration function by $PDO(n).$  Thus, from the example above, we see that $PDO(4)=5,$ where we have counted the following five objects:  
$$
3'+1',\ \ \  1'+1+1+1,\ \ \  1+1'+1+1,\ \ \  1+1+1'+1,\ \ \  1+1+1+1'
$$
Since then, numerous authors have studied arithmetic properties of $PD(n)$ as well as the family of functions $PD_k(n)$ which denotes the number of $k$--regular partitions of $n$ with designated summands.  (Recall that a partition is called {\it $k$--regular} if none of the parts of the partition are divisible by $k.$)  Note that $PDO(n) = PD_2(n)$ in this notation. 

Beginning with \cite{ALL}, a wide variety of Ramanujan--like congruences have been proven for $PD(n)$ and $PD_k(n)$ for infinitely many values of $k$.  See \cite{BK, BO, CJJS, dSS3, dSSk, HBN, NG, Xia} for such work.  

Recently, Herden, et. al. \cite{Her} proved a number of arithmetic properties satisfied by the functions $PD_{2^\ell}(n)$ and $PD_{3^\ell}(n)$ for a variety of values of $\ell$.  At the end of their paper, they shared the following conjecture which will serve as the starting point for the work in this paper:

    \begin{conjecture}
    \label{Herden_conjectures}
        For $n\geq 0,$ we have 
        \begin{eqnarray*}
        PD_2(16n+12) &\equiv & 0 \pmod{4}, \\
        PD_2(24n+20) &\equiv & 0 \pmod{4}, \\
        PD_2(25n+5) &\equiv & 0 \pmod{4}, \\
        PD_2(32n+24) &\equiv & 0 \pmod{4}, \\
        PD_2(48n+26) &\equiv & 0 \pmod{4}          
        \end{eqnarray*}
    \end{conjecture}
It is intriguing to note that the fourth arithmetic progression which appears above, $32n+24$, equals $2(16n+12)$, i.e., $32n+24$ is twice the first arithmetic progression above.  This is a hint of something much larger; indeed, these two congruences are true, and they belong to an easily--described infinite family of congruences modulo 4:  

\begin{theorem}
\label{infinite_family_mod4}
For all $\alpha \geq 0$ and all $n\geq 0,$
$$PD_2(2^\alpha(4n+3)) \equiv 0 \pmod{4}.$$
\end{theorem}
In fact, there is an additional family of congruences of a similar form modulo 8 which we will also prove below.  

\begin{theorem}
\label{infinite_family_mod8}
For all $\alpha \geq 0$ and all $n\geq 0,$
$$PD_2(2^\alpha(8n+7)) \equiv 0 \pmod{8}.$$  
\end{theorem}

Our main goal in this paper is to prove Theorems \ref{infinite_family_mod4} and \ref{infinite_family_mod8}.  All of the proof techniques used herein are elementary, relying on classical $q$--series  identites and generating function manipulations.


\section{Preliminaries}

As noted in \cite{ALL}, the generating function for $PD_2(n)$ is given by 
\begin{equation}
    \label{genfn_main}
    \sum_{n=0}^{\infty} PD_2(n)q^n = \frac{f_4f_6^2}{f_1f_3f_{12}}
\end{equation}
where $f_r = (1-q^r)(1-q^{2r})(1-q^{3r})(1-q^{4r})\dots$ is the usual $q$--Pochhammer symbol.

In order to prove our results, we will require several elementary generating function dissection tools.  For the most part, these will involve well--known 2--dissection results that will allow us to manipulate (\ref{genfn_main}) to our advantage.  We share these results here.  

\begin{lemma}
\label{L4}
We have
\begin{align*}
\frac{1}{f_1^4} & = \frac{f_4^{14}}{f_2^{14}f_{8}^4} + 4q \frac{f_4^2f_{8}^4}{f_2^{10}}. \\
\end{align*}
\end{lemma}

\begin{proof}
This identity can be found in the work of Brietzke, da Silva, and Sellers \cite[Equation (18)]{BSS}. 
\end{proof}

\begin{lemma}
\label{L12}
We have
$$
f_1^2  = \frac{f_2f_8^5}{f_4^2f_{16}^2} - 2q \frac{f_2f_{16}^2}{f_8}.
$$
\end{lemma}

\begin{proof}
See the work of da Silva and Sellers \cite[Equation (12), Lemma 1]{dSS}. 
\end{proof}

\begin{lemma}
We have
\begin{align*}
\frac{1}{f_1f_3} & = \frac{f_8^2f_{12}^5}{f_2^2f_4f_6^4f_{24}^2} + q \frac{f_4^5f_{24}^2}{f_2^4f_6^2f_8^2f_{12}}, \\  
f_1f_3 & = \frac{f_2f_8^2f_{12}^4}{f_4^2f_6f_{24}^2} - q \frac{f_4^4f_6f_{24}^2}{f_2f_8^2f_{12}^2}. 
\end{align*}
\label{L3}
\end{lemma}

\begin{proof}
Both of these identities can be found in Hirschhorn \cite{H}; the first is Equation (30.12.3) while the second is Equation (30.12.1).
\end{proof}



Finally, we state without proof the well--known result regarding congruences of the functions $f_k$ modulo powers of a prime (which, in essence, relies on the divisibility properties of binomial coefficients).  
\begin{lemma}
\label{general_congs_result}    
For any prime $p,$ any $j\geq 1,$ and any $m\geq 1,$
$f_{m}^{p^jk} \equiv f_{pm}^{p^{j-1}k} \pmod{p^j}$.
\end{lemma}

We close this section by listing a number of dissection results for the generating function for $PD_2$ which are easily found via elementary generating function dissections.  

\begin{theorem}
\label{2-dissections}    
We have 
\begin{eqnarray*}
\sum_{n=0}^{\infty} PD_2(2n)q^n
&=& 
\frac{f_4^2f_6^4}{f_1^2f_3^2f_{12}^2}, \textrm{\ \ \ and} \\
\sum_{n=0}^{\infty} PD_2(2n+1)q^n
&=& 
\frac{f_2^6f_{12}^2}{f_1^4f_4^2f_{6}^2}.
\end{eqnarray*}
\end{theorem}
\begin{proof}
See \cite[Theorem 21]{ALL}.    
\end{proof}
\begin{theorem}
\label{4-dissections}    
We have 
\begin{eqnarray*}
\sum_{n=0}^{\infty} PD_2(4n)q^n
&=& 
\frac{f_4^4f_6^8}{f_1^4f_3^4f_{12}^4} + q\frac{f_2^{12}f_{12}^4}{f_1^8f_4^4f_6^4}, \\
\sum_{n=0}^{\infty} PD_2(4n+1)q^n
&=& 
\frac{f_2^{12}f_6^2}{f_1^8f_3^2f_4^4}, \\
\sum_{n=0}^{\infty} PD_2(4n+2)q^n
&=& 
2\frac{f_2^6f_6^2}{f_1^6f_3^2}, \textrm{\ \ \ and} \\
\sum_{n=0}^{\infty} PD_2(4n+3)q^n
&=& 
4\frac{f_4^4f_6^2}{f_1^4f_3^2}.
\end{eqnarray*}
\end{theorem}
\begin{proof}
Using Lemma \ref{L3} and Theorem \ref{2-dissections},  we have 
\begin{eqnarray}
\sum_{n=0}^{\infty} PD_2(2n)q^n
&=& 
\frac{f_4^2f_6^4}{f_1^2f_3^2f_{12}^2} \notag \\
&=& 
\left( \frac{1}{f_1f_3} \right)^2 \frac{f_4^2f_6^4}{f_{12}^2} \notag \\
&=& 
\left( \frac{f_8^2f_{12}^5}{f_2^2f_4f_6^4f_{24}^2} + q \frac{f_4^5f_{24}^2}{f_2^4f_6^2f_8^2f_{12}} \right)^2 \frac{f_4^2f_6^4}{f_{12}^2} \notag \\
&=& 
\left( \frac{f_8^4f_{12}^{10}}{f_2^4f_4^2f_6^8f_{24}^4} + 2q \frac{f_8^2f_{12}^5}{f_2^2f_4f_6^4f_{24}^2}\cdot \frac{f_4^5f_{24}^2}{f_2^4f_6^2f_8^2f_{12}} + q^2\frac{f_4^{10}f_{24}^4}{f_2^8f_6^4f_8^4f_{12}^2}\right) \frac{f_4^2f_6^4}{f_{12}^2}. \label{2n_eqn1} 
\end{eqnarray}
From (\ref{2n_eqn1}), we see that 
\begin{eqnarray*}
\sum_{n=0}^{\infty} PD_2(4n)q^{2n}
&=& 
 \frac{f_8^4f_{12}^{8}}{f_2^4f_6^4f_{24}^4}  + q^2\frac{f_4^{12}f_{24}^4}{f_2^8f_8^4f_{12}^4}
\end{eqnarray*}
or 
\begin{eqnarray*}
\sum_{n=0}^{\infty} PD_2(4n)q^{n}
&=& 
 \frac{f_4^4f_{6}^{8}}{f_1^4f_3^4f_{12}^4}  + q\frac{f_2^{12}f_{12}^4}{f_1^8f_4^4f_{6}^4}.
\end{eqnarray*}
Also from (\ref{2n_eqn1}), we see that 
\begin{eqnarray*}
\sum_{n=0}^{\infty} PD_2(4n+2)q^{2n+1}
&=& 
2q \frac{f_8^2f_{12}^5}{f_2^2f_4f_6^4f_{24}^2}\cdot \frac{f_4^5f_{24}^2}{f_2^4f_6^2f_8^2f_{12}}  \cdot \frac{f_4^2f_6^4}{f_{12}^2} \\
&=& 
2q \frac{f_4^6f_{12}^2}{f_2^6f_6^2}
\end{eqnarray*}
after simplification.  And this implies 
\begin{eqnarray*}
\sum_{n=0}^{\infty} PD_2(4n+2)q^{n}
&=& 
2\frac{f_2^6f_{6}^2}{f_1^6f_3^2}
\end{eqnarray*}
which also appears in \cite[Theorem 1.3]{BO}.  Indeed, the generating function results for $PD_2(4n+1)$ and $PD_2(4n+3)$ follow in similar fashion and can also be found in \cite[Theorem 1.3]{BO}.
\end{proof}
With these tools in hand, we proceed to our elementary proofs of Theorems \ref{infinite_family_mod4} and \ref{infinite_family_mod8}.

\section{An Infinite Family of Congruences Modulo 4}
Our focus in this section is on Theorem \ref{infinite_family_mod4}.  In order to prove this result, we begin with several generating function dissections for $PD_2(n)$ modulo 4.  
We begin by proving the 2-dissection for $PD_2(n)$ modulo 4.

\begin{theorem}
\label{2n_mod4}
We have 
\begin{eqnarray*}
\sum_{n=0}^{\infty} PD_2(2n)q^n &\equiv & \frac{f_4^2}{f_1^2f_3^2} \pmod{4}, \textrm{\ \ \ and} \\
\sum_{n=0}^{\infty} PD_2(2n+1)q^n &\equiv & f_6^2 \pmod{4}. \\
\end{eqnarray*}
\end{theorem}
\begin{proof}
Thanks to Lemma \ref{general_congs_result} and Theorem \ref{2-dissections}, we see that 
\begin{eqnarray*}
 \sum_{n=0}^{\infty} PD_2(2n)q^{n} 
 &=& \frac{f_4^2f_{6}^4}{f_1^2f_3^2f_{12}^2} \\
 &\equiv & \frac{f_2^4f_{12}^2}{f_1^2f_3^2f_{12}^2} \pmod{4} \\
 &\equiv & \frac{f_4^2}{f_1^2f_3^2} \pmod{4}.
\end{eqnarray*}
Again using Lemma \ref{general_congs_result}, Theorem \ref{2-dissections}, and elementary simplifications, we have 
\begin{eqnarray*}
 \sum_{n=0}^{\infty} PD_2(2n+1)q^{n} 
 &=&  
 \frac{f_2^6f_{12}^2}{f_1^4f_4^2f_{6}^2} \\
 &\equiv &  
 \frac{f_{6}^4}{f_{6}^2} \pmod{4}\\
 &=& 
 f_6^2. 
\end{eqnarray*}
\end{proof}
Thanks to the fact that   $\sum_{n=0}^{\infty} PD_2(2n+1)q^{n} \equiv f_6^2 \pmod{4}$, which is a function of $q^6$, we immediately have the following corollary:  
\begin{corollary}
\label{4n3_6n3_6n5_mod4}
For all $n\geq 0,$ 
\begin{eqnarray*}
PD_2(4n+3) &\equiv 0 & \pmod{4}, \\
PD_2(6n+3) &\equiv 0 & \pmod{4}, \textrm{\ \ \ and} \\
PD_2(6n+5) &\equiv 0 & \pmod{4}.
\end{eqnarray*}   
\end{corollary}
\begin{proof} 
We note that, once expanded, the power series representation of $f_6^2$ must only contain powers of $q$ whose exponents are multiples of 6.  Therefore, we immediately know that $PD_2(2(2n+1)+1),$ $PD_2(2(3n+1)+1),$ and $PD_2(2(3n+2)+1)$ must all be congruent to 0 $\pmod{4}.$ 
\end{proof}
We note that the above congruences, along with several others, appear in \cite[Theorem 1.4]{BO}.  

We next transition to the necessary 4--dissections of the generating function for $PD_2(n)$ modulo 4.

\begin{theorem}
\label{4n_mod4}
We have
\begin{align}
\sum_{n=0}^{\infty} PD_2(4n)q^n & \equiv  \left(\frac{f_2^3}{f_6}\right)^2 +qf_{12}^2 \pmod{4}  \textrm{\ \ \ and}
\label{4diss} \\
\sum_{n=0}^{\infty} PD_2(4n+2)q^n & \equiv 2f_2^3f_6 \pmod{4}. \notag
\end{align}
\end{theorem}

\begin{proof}
From Theorem \ref{2n_mod4}, we see that 
\begin{eqnarray*}
 \sum_{n=0}^{\infty} PD_2(2n)q^{n} 
 &\equiv & \frac{f_4^2}{f_1^2f_3^2} \pmod{4} \\
 &=& f_4^2\left( \frac{1}{f_1f_3} \right)^2.
\end{eqnarray*}
Thus, thanks to Lemma \ref{L3}, we know 
\begin{eqnarray}
 \sum_{n=0}^{\infty} PD_2(2n)q^{n} 
 &\equiv & f_4^2\left( \frac{1}{f_1f_3} \right)^2 \pmod{4} \notag \\
 &=& f_4^2\left( \frac{f_8^2f_{12}^5}{f_2^2f_4f_6^4f_{24}^2} + q \frac{f_4^5f_{24}^2}{f_2^4f_6^2f_8^2f_{12}} \right)^2 \notag \\
 &=& f_4^2\left( \frac{f_8^4f_{12}^{10}}{f_2^4f_4^2f_6^8f_{24}^4} + q^2\frac{f_4^{10}f_{24}^4}{f_2^8f_6^4f_8^4f_{12}^2} + 2q\frac{f_8^2f_{12}^5}{f_2^2f_4f_6^4f_{24}^2}\cdot \frac{f_4^5f_{24}^2}{f_2^4f_6^2f_8^2f_{12}}\right)  \notag \\
 &=& \frac{f_8^4f_{12}^{10}}{f_2^4f_6^8f_{24}^4} + q^2\frac{f_4^{12}f_{24}^4}{f_2^8f_6^4f_8^4f_{12}^2} + 2q\frac{f_4^6f_{12}^4}{f_2^6f_6^6}. \label{4n0_4n2}
\end{eqnarray}
Collecting those terms above which are even functions of $q$ and replacing $q^2$ by $q$, we see that 
\begin{eqnarray*}
\sum_{n=0}^{\infty} PD_2(4n)q^n 
& \equiv & \frac{f_4^4f_{6}^{10}}{f_1^4f_3^8f_{12}^4} + q\frac{f_2^{12}f_{12}^4}{f_1^8f_3^4f_4^4f_{6}^2} \pmod{4} \\
&\equiv & 
\frac{f_2^8f_{6}^{10}}{f_2^2f_6^4f_{6}^8} + q\frac{f_4^{6}f_{6}^8}{f_4^2f_6^2f_4^4f_{6}^2} \pmod{4} \\
&\equiv & 
\frac{f_2^6}{f_6^2} + qf_{6}^4 \pmod{4} \\
&\equiv & 
\left(\frac{f_2^3}{f_6}\right)^2 +qf_{12}^2 \pmod{4} 
\end{eqnarray*}
after simplification.  Returning to (\ref{4n0_4n2}) and focusing on the odd powers of $q$, we see that 
\begin{eqnarray*}
 \sum_{n=0}^{\infty} PD_2(4n+2)q^{2n+1} 
 &\equiv & 2q\frac{f_4^6f_{12}^4}{f_2^6f_6^6} \pmod{4}
 \end{eqnarray*}
which implies 
\begin{eqnarray*}
 \sum_{n=0}^{\infty} PD_2(4n+2)q^{n} 
 &\equiv & 2\frac{f_2^6f_{6}^4}{f_1^6f_3^6} \pmod{4} \\
 &\equiv & 2\frac{f_2^6f_{6}^4}{f_2^3f_6^3} \pmod{4} \\
 &\equiv & 2f_2^3f_6 \pmod{4}
 \end{eqnarray*}
using Lemma \ref{general_congs_result}.  
\end{proof}
Thanks to Theorem \ref{4n_mod4}, we have the following corollaries:  
\begin{corollary}
\label{8n6_mod4}
For all $n\geq 0,$ $PD_2(8n+6) \equiv 0\pmod{4}.$
\end{corollary}
\begin{proof}
From Theorem \ref{4n_mod4} we know 
$$\sum_{n=0}^{\infty} PD_2(4n+2)q^n \equiv 2f_2^3f_6 \pmod{4}$$ where $f_2^3f_6$ is an even function of $q.$  Therefore, for all $n\geq 0,$ 
$$
PD_2(4(2n+1)+2) = PD_2(8n+6) \equiv 0 \pmod{4}.
$$    
\end{proof}
Note that the result in Corollary \ref{8n6_mod4} also appears in \cite[Theorem 1.5]{BO}.  
\begin{corollary}
\label{8n0_8n4_mod4}   
We have 
\begin{eqnarray*}
\sum_{n=0}^{\infty} PD_2(8n)q^n & \equiv &  \left(\frac{f_1^3}{f_3}\right)^2 \pmod{4}  \textrm{\ \ \ and}
 \\
\sum_{n=0}^{\infty} PD_2(8n+4)q^n & \equiv & f_{6}^2 \pmod{4}.\\
\end{eqnarray*}
\end{corollary}
\begin{proof}
Thanks to Theorem \ref{4n_mod4}
we know  
$$
\sum_{n=0}^{\infty} PD_2(4n)q^n  \equiv  \left(\frac{f_2^3}{f_6}\right)^2 +qf_{12}^2 \pmod{4}  
$$
and the right--hand side above is already 2--dissected.  Therefore, the corollary follows by simply splitting the result for the generating function congruence for $PD_2(4n)$ into its even and odd parts and replacing $q^2$ by $q$. 
\end{proof}
Corollary \ref{8n0_8n4_mod4} also leads us to our next Ramanujan--like divisibility property.  
\begin{corollary}
\label{16n12_mod4}
For all $n\geq 0,$ $PD_2(16n+12) \equiv 0 \pmod{4}.$
\end{corollary}
\begin{proof}
From Corollary \ref{8n0_8n4_mod4}, we see that 
$$
\sum_{n=0}^{\infty} PD_2(8n+4)q^n  \equiv f_{6}^2 \pmod{4}.
$$
Because the right--hand side of this congruence is an even function of $q$, we know that, for all $n\geq 0$,  
$$
PD_2(8(2n+1)+4) = PD_2(16n+12)\equiv 0 \pmod{4}.
$$
\end{proof}
Note that the congruence in Corollary \ref{16n12_mod4} was one of the congruences conjectured in \cite{Her}.  

We now need only one additional tool in order to complete our proof of Theorem \ref{infinite_family_mod4}.  The following theorem provides an ``internal congruence'' modulo 4 which is satisfied by $PD_2(n)$.  

\begin{theorem}
\label{internal_cong_mod4}
For all $n\geq 0,$ $PD_2(4n) \equiv PD_2(n) \pmod{4}.$
\end{theorem}
\begin{proof}
From (\ref{4diss}) above, we know that 
$$\sum_{n=0}^\infty PD_2(4n)q^n \equiv  \left(\frac{f_2^3}{f_6}\right)^2 +qf_{12}^2 \pmod{4}.$$  
From \cite[Lemma 4]{Her}, we also know 
$$\sum_{n=0}^\infty PD_2(n)q^n \equiv  \left(\frac{f_2^3}{f_6}\right)^2 +qf_{12}^2 \pmod{4}.$$  The result follows.  
\end{proof}

We now have all of the tools needed to provide a proof of Theorem \ref{infinite_family_mod4}.  
\begin{proof}
(of Theorem \ref{infinite_family_mod4})  Note that the $\alpha=0$ case is proven in Corollary \ref{4n3_6n3_6n5_mod4}, while the $\alpha = 1$ case appears in Corollary \ref{8n6_mod4}.   We then see that, thanks to Theorem \ref{internal_cong_mod4}, we have (mod 4)
$$
0 \equiv PD_2(4n+3) \equiv PD_2(4(4n+3)) \equiv PD_2(4^2(4n+3)) \equiv PD_2(4^3(4n+3)) \dots
$$
which takes care of all of the cases with an even power of $\alpha.$  For the cases where $\alpha$ is odd, we note that  $8n+6 = 2(4n+3),$ and this then gives us (mod 4) 
$$
0 \equiv PD_2(2(4n+3)) \equiv PD_2(4\cdot 2(4n+3)) \equiv PD_2(4^2\cdot 2(4n+3)) \equiv  PD_2(4^3\cdot 2(4n+3)) \dots
$$
\end{proof}
We now have a proof of an infinite family of congruences modulo 4 satisfied by $PD_2$ which encompasses two of the conjectured congruences in \cite{Her}.  

We close this section by noting that, thanks to Corollary \ref{4n3_6n3_6n5_mod4} and Theorem \ref{internal_cong_mod4}, we now know the following:  
\begin{theorem}
\label{infinite_family_mod4_6n5}
For all $\alpha \geq 0$ and all $n\geq 0,$
$$PD_2(4^\alpha(6n+5)) \equiv 0 \pmod{4}.$$
\end{theorem}
This is another satisfying infinite family of congruences modulo 4, and we note that the $\alpha=1$ case of this family is the second conjecture that appeared in the list at the end of \cite{Her}.  

\section{An Infinite Family of Congruences Modulo 8}
We now transition to the proof of a related family of Ramanujan--like congruences satisfied by $PD_2,$ this time modulo 8, which was stated in Theorem \ref{infinite_family_mod8}.  As a reminder, that theorem states that, for all $\alpha \geq 0$ and all $n\geq 0,$
$$PD_2(2^\alpha(8n+7)) \equiv 0 \pmod{8}.$$  
We will prove Theorem \ref{infinite_family_mod8} in an analogous manner to that used above; namely, we will directly prove a few ``base'' cases of the congruence, and then we will prove an internal congruence modulo 8 which will serve as the mechanism by which we can allow $\alpha $ to grow arbitrarily large.  

To begin this process, we prove the $\alpha =0$ case of the above theorem.  
\begin{theorem}
\label{8n7_mod8}
 For all $n\geq 0,$ $PD_2(8n+7) \equiv 0 \pmod{8}.$   
\end{theorem}
\begin{proof}
Thanks to Theorem \ref{2-dissections} as well as Lemma \ref{L4}, we see that 
\begin{eqnarray*}
\sum_{n=0}^{\infty} PD_2(2n+1)q^{n} 
&=&
\frac{1}{f_1^4}\left( \frac{f_2^6f_{12}^2}{f_4^2f_{6}^2}  \right)  \\
&=&
\left( \frac{f_4^{14}}{f_2^{14}f_8^4} + 4q\frac{f_4^2f_8^4}{f_2^{10}} \right)\left( \frac{f_2^6f_{12}^2}{f_4^2f_{6}^2}  \right).
\end{eqnarray*}
Thus, we see that 
$$
\sum_{n=0}^{\infty} PD_2(4n+3)q^{2n+1} = 4q\frac{f_8^4f_{12}^2}{f_2^4f_6^2} 
$$
or 
$$
\sum_{n=0}^{\infty} PD_2(4n+3)q^{n} = 4\frac{f_4^4f_{6}^2}{f_1^4f_3^2} \equiv  4\frac{f_4^4f_{6}^2}{f_2^2f_6} \pmod{8}.
$$
Because the function $\frac{f_4^4f_{6}^2}{f_2^2f_6} = \frac{f_4^4f_{6}}{f_2^2} $ is an even function of $q$, we immediately conclude that, for all $n\geq 0,$ $PD_2(4(2n+1)+3) = PD_2(8n+7) \equiv 0 \pmod{8}.$
\end{proof}
Next, we prove the $\alpha=1$ case of Theorem \ref{infinite_family_mod8}.

\begin{theorem}
\label{16n14_mod8}
 For all $n\geq 0,$ $PD_2(16n+14) \equiv 0 \pmod{8}.$   
\end{theorem}
\begin{proof}
Using Theorem \ref{4-dissections} and Lemmas \ref{L3} and \ref{L4}, we have 
\begin{eqnarray*}
\sum_{n=0}^{\infty} PD_2(4n+2)q^{n} 
&=&
2\frac{f_2^6f_6^2}{f_1^6f_3^2}  \\
&=& 
2f_2^6f_6^2\left( \frac{1}{f_1^4} \right)\left(\frac{1}{f_1f_3}\right)^2 \\
&=&
2f_2^6f_6^2\left( \frac{f_4^{14}}{f_2^{14}f_{8}^4} + 4q \frac{f_4^2f_{8}^4}{f_2^{10}} \right)\left( \frac{f_8^2f_{12}^5}{f_2^2f_4f_6^4f_{24}^2} + q \frac{f_4^5f_{24}^2}{f_2^4f_6^2f_8^2f_{12}}  \right)^2 \\
&\equiv &
2\left( \frac{f_4^{14}f_6^2}{f_2^{8}f_{8}^4}  \right)\left( \frac{f_8^2f_{12}^5}{f_2^2f_4f_6^4f_{24}^2} + q \frac{f_4^5f_{24}^2}{f_2^4f_6^2f_8^2f_{12}}  \right)^2 \pmod{8}.
\end{eqnarray*}
Expanding the above and keeping only those terms where the powers of $q$ are odd yields 
\begin{eqnarray*}
\sum_{n=0}^{\infty} PD_2(8n+6)q^{2n+1} 
&\equiv &
2\left( \frac{f_4^{14}f_6^2}{f_2^{8}f_{8}^4}  \right)\left(2q \frac{f_8^2f_{12}^5}{f_2^2f_4f_6^4f_{24}^2} \cdot \frac{f_4^5f_{24}^2}{f_2^4f_6^2f_8^2f_{12}}  \right) \pmod{8} \\
&\equiv &
4q\frac{f_4^{18}f_{12}^4}{f_2^{14}f_6^4f_8^4} \pmod{8}
\end{eqnarray*}    
which implies that 
\begin{eqnarray*}
\sum_{n=0}^{\infty} PD_2(8n+6)q^{n} 
&\equiv &
4\frac{f_2^{18}f_{6}^4}{f_1^{14}f_3^4f_4^4} \pmod{8} \\
&\equiv &
4\frac{f_2^{11}f_{6}^2}{f_4^4} \pmod{8}.
\end{eqnarray*}    
Since the last expression above is an even function of $q$, we immediately know that, for all $n\geq 0$, 
$$PD_2(8(2n+1)+6) = PD_2(16n+14) \equiv 0\pmod{8}.$$
    
\end{proof}
We next prove the $\alpha=2$ case of Theorem \ref{infinite_family_mod8}.  

\begin{theorem}
\label{32n28_mod8}
 For all $n\geq 0,$ $PD_2(32n+28) \equiv 0 \pmod{8}.$   
\end{theorem}

\begin{proof}
Thanks to Theorem \ref{4-dissections}, we know 
\begin{eqnarray}
\sum_{n=0}^\infty PD_2(4n)q^n 
&=& 
\frac{f_4^4f_6^8}{f_1^4f_3^4f_{12}^4} + q\frac{f_2^{12}f_{12}^4}{f_1^8f_4^4f_6^4} \notag \\
&\equiv & 
\left( \frac{f_4}{f_1f_3} \right)^4 + q\left( \frac{f_{12}}{f_6} \right)^4 \pmod{8} \ \ \textrm{using Lemma \ref{general_congs_result}} \label{4n_mod8_overall}\\
&=& 
f_4^4\left( \frac{f_8^2f_{12}^5}{f_2^2f_4f_6^4f_{24}^2} + q \frac{f_4^5f_{24}^2}{f_2^4f_6^2f_8^2f_{12}}  \right)^4 + q\frac{f_{12}^4}{f_6^4}. \label{4n_mod8}
\end{eqnarray}
Thus, 
\begin{eqnarray*}
&&
\sum_{n=0}^\infty PD_2(8n+4)q^{2n+1} \\
&\equiv &  
f_4^4\left( 4q\left( \frac{f_8^2f_{12}^5}{f_2^2f_4f_6^4f_{24}^2}\right)^3\left( \frac{f_4^5f_{24}^2}{f_2^4f_6^2f_8^2f_{12}} \right)\right) \\
&& + f_4^4\left( 4q^3\left( \frac{f_8^2f_{12}^5}{f_2^2f_4f_6^4f_{24}^2}\right)\left( \frac{f_4^5f_{24}^2}{f_2^4f_6^2f_8^2f_{12}} \right)^3\right) + q\frac{f_{12}^4}{f_6^4} \pmod{8} \\
&\equiv &  
4q\frac{f_4^9f_8^6f_{12}^{15}f_{24}^2}{f_2^{10}f_4^3f_6^{14}f_8^2f_{12}f_{24}^6}  + 4q^3 \frac{f_4^{19}f_8^2f_{12}^5f_{24}^6}{f_2^{14}f_4f_6^{10}f_8^6f_{12}^3f_{24}^2} + q\frac{f_{12}^4}{f_6^4} \pmod{8}.
\end{eqnarray*}
This means 
\begin{eqnarray*}
&&
\sum_{n=0}^\infty PD_2(8n+4)q^{n} \\
&\equiv &  
4\frac{f_2^9f_4^6f_{6}^{15}f_{12}^2}{f_1^{10}f_2^3f_3^{14}f_4^2f_6f_{12}^{6}}  + 4q \frac{f_2^{19}f_4^2f_{6}^5f_{12}^6}{f_1^{14}f_2f_3^{10}f_4^6f_{6}^3f_{12}^2} + \frac{f_{6}^4}{f_3^4} \pmod{8} \\
&\equiv & 
4\frac{f_2^9f_2^{12}f_{6}^{15}f_{6}^4}{f_2^{5}f_2^3f_6^{7}f_2^4f_6f_{6}^{12}}  + 4q \frac{f_2^{19}f_2^4f_{6}^5f_{6}^{12}}{f_2^{7}f_2f_6^{5}f_2^{12}f_{6}^3f_{6}^4} + \frac{f_{6}^4}{f_3^4} \pmod{8} \\
&\equiv & 
4\frac{f_2^9}{f_6} + 4qf_2^3f_6^5 + f_6^4\left( \frac{f_{12}^{14}}{f_6^{14}f_{24}^4} + 4q^3 \frac{f_{12}^2f_{24}^4}{f_6^{10}} \right)  \pmod{8}
\end{eqnarray*}
using Lemma \ref{L4} with $q$ replaced by $q^3.$  Retaining only those terms where the powers of $q$ are odd yields 
\begin{eqnarray*}
\sum_{n=0}^\infty PD_2(16n+12)q^{2n+1} 
&\equiv & 
4qf_2^3f_6^5 + 4q^3 \frac{f_6^4f_{12}^2f_{24}^4}{f_6^{10}}   \pmod{8}
\end{eqnarray*}
which means 
\begin{eqnarray*}
\sum_{n=0}^\infty PD_2(16n+12)q^{n} 
&\equiv & 
4f_1^3f_3^5 + 4q \frac{f_{6}^2f_{12}^4}{f_3^{6}}   \pmod{8} \\
&\equiv &
4f_1^2f_1f_3^4f_3 + 4q \frac{f_{6}^2f_{6}^8}{f_6^{3}}   \pmod{8} \\
&\equiv &
4f_1f_2f_3f_6^2 + 4q f_6^7   \pmod{8} \\
&\equiv & 
4f_2f_6^2\left( \frac{f_2f_8^2f_{12}^4}{f_4^2f_6f_{24}^2} - q \frac{f_4^4f_6f_{24}^2}{f_2f_8^2f_{12}^2} \right)  + 4q f_6^7   \pmod{8} \\
&\equiv & 
4\frac{f_2^2f_6f_8^2f_{12}^4}{f_4^2f_{24}^2} +4q \frac{f_4^4f_6^3f_{24}^2}{f_8^2f_{12}^2}  + 4q f_6^7   \pmod{8} \\
\end{eqnarray*}
where in the penultimate line above we have used Lemma \ref{L3}.
This implies that 
\begin{eqnarray*}
\sum_{n=0}^\infty PD_2(32n+28)q^{2n+1} 
&\equiv & 
4q \frac{f_4^4f_6^3f_{24}^2}{f_8^2f_{12}^2}  + 4q f_6^7   \pmod{8} 
\end{eqnarray*}
or 
\begin{eqnarray*}
\sum_{n=0}^\infty PD_2(32n+28)q^{n} 
&\equiv & 
4\frac{f_2^4f_3^3f_{12}^2}{f_4^2f_{6}^2}  + 4 f_3^7   \pmod{8} \\
&\equiv & 
4\frac{f_4^2f_3^3f_{3}^8}{f_4^2f_{3}^4}  + 4 f_3^7   \pmod{8} \\
&\equiv & 
4 f_3^7   + 4 f_3^7   \pmod{8} \\
&=& 
0.
\end{eqnarray*}
The result follows. 
\end{proof}
We now prove one more Ramanujan--like congruence modulo 8 which is the $\alpha=3$ case of Theorem \ref{infinite_family_mod8}.  

\begin{theorem}
\label{64n56_mod8}
For all $n\geq 0,$ $PD_2(64n+56)\equiv 0 \pmod{8}.$    
\end{theorem}
\begin{proof}
Thanks to (\ref{4n_mod8}), we see that
\begin{eqnarray*}
\sum_{n=0}^\infty PD_2(8n)q^{2n} 
&\equiv & 
\frac{f_4^4f_8^8f_{12}^{20}}{f_2^8f_4^4f_6^{16}f_{24}^8} + 6q^2\frac{f_4^4f_8^4f_{12}^{10}}{f_2^4f_4^2f_6^8f_{24}^4} \cdot \frac{f_4^{10}f_{24}^4}{f_2^8f_6^4f_8^4f_{12}^2} \\
&\ \ \ & + q^4\frac{f_4^{4}f_4^{20}f_{24}^8}{f_2^{16}f_6^8f_8^8f_{12}^4} \pmod{8}.
\end{eqnarray*}
Thus, 
\begin{eqnarray}
\sum_{n=0}^\infty PD_2(8n)q^{n} 
&\equiv & 
\frac{f_4^8f_{6}^{20}}{f_1^8f_3^{16}f_{12}^8} + 6q\frac{f_2^{14}f_{6}^{10}}{f_1^{12}f_2^2f_3^{12}f_{6}^2} +q^2\frac{f_2^{24}f_{12}^8}{f_1^{16}f_3^8f_4^8f_{6}^4} \pmod{8} \notag \\
&\equiv & 
\frac{f_8^4f_{6}^{20}}{f_2^4f_6^{8}f_{6}^{16}} + 6q\frac{f_2^{14}f_{6}^{10}}{f_2^{6}f_2^2f_6^{6}f_{6}^2} +q^2\frac{f_2^{24}f_{12}^8}{f_2^{8}f_6^4f_4^8f_{6}^4} \pmod{8} \notag \\
&\equiv & 
\frac{f_8^4}{f_2^4f_6^{4}} + 6q\frac{f_2^{14}f_{6}^{10}}{f_2^{8}f_6^{8}} +q^2\frac{f_2^{24}f_{24}^4}{f_2^{8}f_{12}^4f_4^8} \pmod{8}. \label{8n_mod8}
\end{eqnarray}
This means that 
\begin{eqnarray*}
\sum_{n=0}^\infty PD_2(16n+8)q^{2n+1} 
&\equiv & 
6q\frac{f_2^{14}f_{6}^{10}}{f_2^{8}f_6^{8}}  \pmod{8} \\
&\equiv & 
6q f_2^6f_6^2 \pmod{8}.
\end{eqnarray*}
Dividing both sides of this congruence by $q$ and replacing $q^2$ by $q$ yields 
\begin{eqnarray*}
\sum_{n=0}^\infty PD_2(16n+8)q^{n} 
&\equiv & 
6 f_1^6f_3^2 \pmod{8} \\
&= & 
6 (f_1^2)^2(f_1f_3)^2.
\end{eqnarray*}
We can then apply Lemma \ref{L12} and Lemma \ref{L3} to get 
\begin{eqnarray*}
&& \sum_{n=0}^\infty PD_2(16n+8)q^{n} \\
\ \ \ \ &\equiv & 
6 (f_1^2)^2(f_1f_3)^2 \pmod{8} \\
\ \ \ \ &\equiv & 
6\left( \frac{f_2f_8^5}{f_4^2f_{16}^2} -2q\frac{f_2f_{16}^2}{f_8}\right)^2 \left( \frac{f_2f_8^2f_{12}^4}{f_4^2f_6f_{24}^2} - q \frac{f_4^4f_6f_{24}^2}{f_2f_8^2f_{12}^2}\right)^2 \pmod{8} \\
\ \ \ \ &\equiv & 
6\left( \frac{f_2f_8^5}{f_4^2f_{16}^2} \right)^2 \left( \frac{f_2^2f_8^4f_{12}^8}{f_4^4f_6^2f_{24}^4} - 2q \frac{f_2f_8^2f_{12}^4}{f_4^2f_6f_{24}^2}\cdot \frac{f_4^4f_6f_{24}^2}{f_2f_8^2f_{12}^2} + q^2\frac{f_4^8f_6^2f_{24}^4}{f_2^2f_8^4f_{12}^4}\right) \pmod{8}.
\end{eqnarray*}
Therefore, 
\begin{eqnarray*}
\sum_{n=0}^\infty PD_2(32n+24)q^{2n+1} 
&\equiv & 
-12q\left( \frac{f_2f_8^5}{f_4^2f_{16}^2} \right)^2 \frac{f_2f_8^2f_{12}^4}{f_4^2f_6f_{24}^2}\cdot \frac{f_4^4f_6f_{24}^2}{f_2f_8^2f_{12}^2}  \pmod{8} \\
&\equiv & 
4q\frac{f_2^2f_8^{10}f_{12}^4}{f_{4}^2f_{12}^2f_{16}^4}  \pmod{8} \\
&\equiv & 
4q\frac{f_{12}^2f_{16}}{f_{4}}  \pmod{8} 
\end{eqnarray*}
after simplification.  Therefore, 
$$\sum_{n=0}^\infty PD_2(32n+24)q^{2n+1} 
\equiv 
4\frac{f_{6}^2f_{8}}{f_{2}}  \pmod{8}.
$$
Because the above is an even function of $q,$ we know that, for all $n\geq 0,$ $$PD_2(32(2n+1)+24) = PD_2(64n+56) \equiv 0 \pmod{8}.$$
\end{proof}
The last item that we need in order to complete the proof of Theorem \ref{infinite_family_mod8} is the following internal congruence.  
\begin{theorem}
\label{internal_cong_mod8}
For all $n\geq 0,$ $PD_2(16n) \equiv PD_2(4n) \pmod{8}.$
\end{theorem}
\begin{proof}
Thanks to (\ref{8n_mod8}), we see that 
\begin{eqnarray*}
\sum_{n=0}^\infty PD_2(16n)q^{2n} 
&\equiv & 
\frac{f_8^4}{f_2^4f_6^{4}} +q^2\frac{f_2^{24}f_{24}^4}{f_2^{8}f_{12}^4f_4^8} \pmod{8} \\
&\equiv & 
\frac{f_8^4}{f_2^4f_6^{4}} +q^2\frac{f_2^{24}f_{24}^4}{f_2^{8}f_{12}^4f_2^{16}} \pmod{8} \\
&\equiv & 
\frac{f_8^4}{f_2^4f_6^{4}} +q^2\frac{f_{24}^4}{f_{12}^4} \pmod{8} 
\end{eqnarray*}
which means 
\begin{eqnarray*}
\sum_{n=0}^\infty PD_2(16n)q^{n} 
&\equiv & 
\frac{f_4^4}{f_1^4f_3^{4}} +q\frac{f_{12}^4}{f_{6}^4} \pmod{8} \\
&\equiv & 
\sum_{n=0}^\infty PD_2(4n)q^{n}  \textrm{\ \ thanks to (\ref{4n_mod8_overall})}.
\end{eqnarray*}
\end{proof}


We now have all of the tools needed to provide a proof of Theorem \ref{infinite_family_mod8}.  
\begin{proof}
(of Theorem \ref{infinite_family_mod8})  Note that the $\alpha=0, 1, 2, 3$ cases are proven in Theorems \ref{8n7_mod8}, \ref{16n14_mod8},  \ref{32n28_mod8}, and \ref{64n56_mod8} respectively.   We then see that, thanks to Theorems \ref{32n28_mod8} and \ref{internal_cong_mod8}, and the fact that $32n+28 = 4(8n+7)$, we have (mod 8)
$$
0 \equiv PD_2(4(8n+7)) \equiv PD_2(4^2(8n+7)) \equiv PD_2(4^3(8n+7)) \dots
$$
which takes care of all of the cases with an even power of $\alpha.$  For the cases where $\alpha$ is odd, we use Theorems \ref{64n56_mod8} and \ref{internal_cong_mod8} and the fact that  
$$64n+56 = 4(2(8n+7))$$ to yield (mod 8) 
$$
0 \equiv PD_2(4(2(8n+7))) \equiv PD_2(4^2(2(8n+7))) \equiv  PD_2(4^3(2(8n+7))) \dots
$$
\end{proof}

\section{Closing Thoughts}
We close by providing an elementary proof of the last congruence in Conjecture \ref{Herden_conjectures}:

\begin{theorem}
For all $n\geq 0,$ $PD_2(48n+26) \equiv 0 \pmod{4}.$    
\end{theorem}
\begin{proof}
From Theorem \ref{4n_mod4}, we know 
$$
\sum_{n=0}^{\infty} PD_2(4n+2)q^n  \equiv 2f_2^3f_6 \pmod{4}. 
$$ 
Thus, we see that 
$$
\sum_{n=0}^{\infty} PD_2(8n+2)q^n  \equiv 2f_1^3f_3 \pmod{4}.
$$ 
Using Lemmas \ref{L12} and \ref{L3}, we then have 
\begin{eqnarray*}
\sum_{n=0}^{\infty} PD_2(8n+2)q^n  
&\equiv & 
2f_1^3f_3 \pmod{4} \\
&= &
2(f_1f_3)f_1^2 \\
&=& 
2\left( \frac{f_2f_8^2f_{12}^4}{f_4^2f_6f_{24}^2} - q \frac{f_4^4f_6f_{24}^2}{f_2f_8^2f_{12}^2} \right)\left( \frac{f_2f_8^5}{f_4^2f_{16}^2} - 2q \frac{f_2f_{16}^2}{f_8} \right)^2 \\
&\equiv & 
2\left( \frac{f_2f_8^2f_{12}^4}{f_4^2f_6f_{24}^2} - q \frac{f_4^4f_6f_{24}^2}{f_2f_8^2f_{12}^2} \right)\left( \frac{f_2f_8^5}{f_4^2f_{16}^2}\right)^2 \pmod{4} \\
&\equiv & 
2\frac{f_2^2f_8^7f_{12}^4}{f_4^4f_6f_{24}^2} - 2q \frac{f_2f_4^4f_6f_8^5f_{24}^2}{f_2f_4^2f_8^2f_{12}^2f_{16}^2}  \pmod{4}.
\end{eqnarray*}
This implies 
\begin{eqnarray*}
\sum_{n=0}^{\infty} PD_2(16n+10)q^{2n+1}  
&\equiv & 
2q \frac{f_2f_4^4f_6f_8^5f_{24}^2}{f_2f_4^2f_8^2f_{12}^2f_{16}^2}  \pmod{4}
\end{eqnarray*}
or 
\begin{eqnarray*}
\sum_{n=0}^{\infty} PD_2(16n+10)q^{n}  
&\equiv & 
2\frac{f_2^4f_3f_4^5f_{12}^2}{f_2^2f_4^2f_{6}^2f_{8}^2}  \pmod{4} \\
&\equiv & 
2\frac{f_2^4f_3f_2^{10}f_{6}^4}{f_2^2f_2^4f_{6}^2f_{2}^8}  \pmod{4} \\
&\equiv & 
2\frac{f_3f_{6}^4}{f_{6}^2}  \pmod{4} \\
&=& 
2f_3f_6^2.
\end{eqnarray*}
Since the last function above is a function of $q^3,$ we know that, for all $n\geq 0,$
$$PD_2(16(3n+1)+10) = PD_2(48n+26) \equiv 0 \pmod{4}.$$
\end{proof}




\begin{thebibliography}{00}





\bibitem{ALL} G. E. Andrews, R. P. Lewis, and J. Lovejoy, 
Partitions with designated summands, {\it Acta Arith.} {\bf 105} (2002), no.1, 51–66.

\bibitem{BK}  N. D. Baruah and M. Kaur, A note on some recent results of da Silva and Sellers on congruences for $k$-regular partitions with designated summands, {\it INTEGERS} {\bf 20} (2020), Paper No. A74.


\bibitem{BO}  N. D. Baruah and K. K. Ojah, Partitions with designated summands in which all parts are odd, {\it INTEGERS} {\bf 15} (2015), A9.


\bibitem{BSS} E. H. M. Brietzke, R. da Silva, J. A. Sellers, Congruences related to an eighth order mock theta function of Gordon and McIntosh, {\it J. Math. Anal. Appl.} {\bf 479} (2019), 62--89.

\bibitem{BLM}  K. Bringmann, J. Lovejoy, and K. Mahlburg, A partition identity and the universal mock theta function $g_2$, {\it Math. Res. Lett.} {\bf 23} (2016), no. 1, 67–-80.

\bibitem{CJJS}  
W.Y.C. Chen, K. Q. Ji, H.-T. Jin, and E.Y.Y. Shen, On the number of partitions with designated summands, {\it J. Number Theory} {\bf 133} (2013), 2929--2938.  

\bibitem{dSS3} 
R. da Silva and J. A. Sellers, New congruences for 3--regular partitions with designated summands, {\it INTEGERS} {\bf 20A} (2020), \#A6. 

\bibitem{dSSk}
R. da Silva and J. A. Sellers, Infinitely many congruences for $k$--regular partitions with designated summands, {\it Bull. Braz. Math. Soc}, New Series {\bf 51} (2020), 357--370.

\bibitem{dSS}
R. da Silva and J. A. Sellers, Congruences for the Coefficients of the Gordon and McIntosh Mock Theta Function $\xi(q)$, {\it Ramanujan Journal} {\bf 58}, no. 3 (2022), 815-–834.

\bibitem{HBN}
B. Hemanthkumar,  H. S. Sumanth Bharadwaj, and M. S. Mahadeva Naika, Congruences modulo small powers of 2 and 3 for partitions into odd designated summands, {\it J. Integer Seq.} {\bf 20} (2017), no. 4, Article 17.4.3.

\bibitem{Her} D. Herden, M. R. Sepanski, J. Stanfill, C. C. Hammon, J. Henningsen, H. Ickes, and I. Ruiz, Partitions with Designated Summands Not Divisible by $2^\ell$, $2,$ and $3^\ell$ Modulo $2,$ $4,$ and $3$, {\it INTEGERS} {\bf 23} (2023), A43.


\bibitem{H} M. D. Hirschhorn, {The power of $q$, a personal journey}, Developments in Mathematics, v. 49, Springer (2017).

\bibitem{NG}
M. S. Mahadeva Naika and D. S. Gireesh, Congruences for 3--regular partitions with designated summands, {\it INTEGERS} {\bf 16} (2016), \#A25.  

\bibitem {SS} D.D.M. Sang and D.Y.H. Shi, Combinatorial proofs and generalization of Bringmann, Lovejoy and Mahlburg’s overpartition theorems, Ramanujan J. {\bf 55} (2021), 539-–554.

\bibitem{Xia} 
E. X. W. Xia, Arithmetic properties of partitions with designated summands, {\it J. Number Theory} {\bf 159} (2016), 160--175.  

\end{thebibliography}
\end{document}